\documentclass[11pt]{amsart}
\usepackage{graphicx}
\usepackage{amscd}
\usepackage{amsmath}
\usepackage{amsfonts}
\usepackage{amssymb}
\usepackage[colorlinks=true]{hyperref}
\usepackage{mathrsfs}
\newtheorem{theorem}{Theorem}[section]
\newtheorem{corollary}[theorem]{Corollary}
\newtheorem{lemma}[theorem]{Lemma}
\newtheorem{proposition}[theorem]{Proposition}
\newtheorem{definition}{Definition}[section]

\newtheorem{step}{Step}

\numberwithin{equation}{section}

\numberwithin{equation}{section}
\allowdisplaybreaks

\begin{document}
\title[Infinitely many solutions for a critical and fractional Kirchhoff problem]
{Infinitely many solutions for\\ a critical Kirchhoff type problem\\ involving a fractional operator}
\author[A.~Fiscella]{Alessio Fiscella}
\address{Departamento de Matem\'atica,
Universidade Estadual de Campinas, IMECC,
Rua S\'ergio Buarque de Holanda, 651, Campinas, SP CEP 13083--859 BRAZIL}
\email{fiscella@ime.unicamp.br}


\begin{abstract}
In this paper we deal with a Kirchhoff type problem driven by a nonlocal fractional 
integrodifferential operator $\mathcal L_K$, that is
$$ -M(\left\|u\right\|^2)\mathcal L_Ku=\lambda f(x,u)\left[\int_\Omega F(x,u(x))dx\right]^r+\left|u\right|^{2^* -2}u\quad \mbox{in }\Omega,\qquad u=0\quad\mbox{in }\mathbb{R}^{n}\setminus\Omega,
$$
where $\Omega$ is an open bounded subset of $\mathbb{R}^n$, $M$ and $f$ are continuous functions, $\left\|\cdot\right\|$ is a functional norm, $F(x,u(x))=\int^{u(x)}_0 f(x,\tau)d\tau$, $2^*$ is a fractional Sobolev exponent, $\lambda$ and $r$ are real parameters. For this problem we prove the existence of infinitely many solutions, through a suitable truncation argument and exploiting the genus theory introduced by Krasnoselskii.
\end{abstract}

\maketitle
\section{Introduction}
In the last years, the interest towards nonlinear Kirchhoff type problems has grown more and more,
thanks in particular to their intriguing analytical structure due to the presence of the nonlocal Kirchhoff function $M$
which makes the equation no longer a pointwise identity.
In the present paper we consider the problem
\begin{align}\label{P}
-&M\left(\|u\|^2\right)\mathcal L_Ku=\lambda f(x,u)\left[\int_\Omega F(x,u(x))dx\right]^r\!+\!\left|u\right|^{2^* -2}\!u\quad\mbox{in } \Omega,\nonumber\\
&u=0\quad\mbox{in } \mathbb{R}^{n}\setminus\Omega,\\
&\|u\|^2=\iint_{\mathbb{R}^{2n}}\!\!\!\! |u(x)-u(y)|^2K(x-y)dxdy,\quad F(x,u(x))=\int^{u(x)}_0 f(x,\tau)d\tau\nonumber
\end{align}
where $\Omega\subset \mathbb R^n$ is a bounded domain, $n>2s$,  with $s\in(0,1)$,  the number
$2^*=2n/(n-2s)$ is the critical exponent of the fractional Sobolev space $H^s(\mathbb{R}^n)$,
$M$ and $f$ are two continuous functions whose properties will be introduced
later, $\lambda>0 $ and $r\geq0$ are real parameters.

The main nonlocal fractional operator $\mathcal L_K$ is defined for any $x\in\mathbb{R}^{n}$ as follows
\begin{equation*}
\mathcal L_K\varphi(x)=\frac{1}{2}\int_{\mathbb{R}^{n}}(\varphi(x+y)+\varphi(x-y)-2\varphi(x))K(y)dy,
\end{equation*}
along any $\varphi\in C^\infty_0(\mathbb R^n)$, where the kernel $K:\mathbb{R}^{n}\setminus\left\{0\right\}\rightarrow \mathbb R^+$ is a
\textit{measurable} function with the properties that
\begin{equation}\label{K1}
\mbox{$mK\in L^1(\mathbb R^n)$, where $m(x)=\min \{|x|^2, 1\}$};
\end{equation}
\begin{equation}\label{K2}
\mbox{there exists }\theta>0
\mbox{ such that }K(x)\geq \theta |x|^{-(n+2s)}\mbox{ for any }x\in \mathbb R^n \setminus\{0\}.
\end{equation}
A typical example for $K$ is given by $K(x)=\frac{2}{c(n,s)}\left|x\right|^{-(n+2s)}$, for which $\mathcal L_K$ coincides with the fractional Laplace operator $-(-\Delta)^s$,
which may be defined by the Riesz potential, for any $x\in\mathbb R^n$, as
\begin{equation*}
-(-\Delta)^s \varphi(x)=
\frac{c(n,s)}{2}
\int_{\mathbb{R}^{n}}\frac{\varphi(x+y)+\varphi(x-y)-2\varphi(x)}{|y|^{n+2s}}dy,
\end{equation*}
along any $\varphi\in C^\infty_0(\mathbb R^n)$, where $c(n,s)>0$ is the normalizing constant given by
\begin{equation}\label{cns}
c(n,s)=\left(\int_{\mathbb R^{n}}\frac{1-\cos(\xi_1)}{|\xi|^{n+2s}}\,d\xi\right)^{-1}
\end{equation}
as defined in \cite{VP} (see this paper and the references therein for further details on fractional Laplacian).

Concerning the Kirchhoff function, we suppose that $M:\mathbb{R}^+_0 \to\mathbb{R}^+_0$ is a \textit{continuous} function verifying
\begin{equation}\label{m1}
\mbox{there exists }m_0> 0\mbox{ such that }M(t)\geq m_0\mbox{ for any }t\in\mathbb{R}^+_0;
\end{equation}
\begin{equation}\label{m2}
\begin{aligned}
&\mbox{there exists }\sigma\in [2,2^*)\mbox{ such that }\frac{1}{2}\mathscr M(t)-\frac{1}{\sigma}M(t)t\geq0\mbox{ for any }t\in\mathbb{R}^+_0,\\
&\mbox{where }\mathscr M(t)=\displaystyle\int_0^t M(\tau)d\tau.
\end{aligned}
\end{equation}
A typical Kirchhoff function verifying \eqref{m1} and \eqref{m2} is given by $M(t)=m_0+m_1 t^{\sigma/2}$, with a suitable constant $m_1\geq0$.

Also, for problem \eqref{P} we assume that $f:\overline{\Omega}\times\mathbb{R}\rightarrow\mathbb{R}$ is a \textit{continuous} function satisfying
\begin{equation}\label{f1}
f(x,-t)=-f(x,t)\mbox{ for any }(x,t)\in\overline{\Omega}\times\mathbb{R};
\end{equation}
\begin{equation}\label{f2}
\begin{aligned}
&\mbox{there exist constants }a_1,\,a_2>0 \mbox{ and }1<q<\frac{2}{r+1}\mbox{ such that}\\
&a_1t^{q-1}\leq f(x,t)\leq a_2t^{q-1}\mbox{ for any }x\in\overline{\Omega}\mbox{ and }t\in[0,\infty).
\end{aligned}
\end{equation}
A very simple model for the nonlinearity $f$ is given by $f(x,t)=a(x)\left|t\right|^{q-2}t$, with $a:\overline{\Omega}\rightarrow\mathbb R$ a continuous and non--negative map. 

We note that $f(x,0)=0$ for any $x\in\overline{\Omega}$, by \eqref{f1}, hence the function $u\equiv0$ is a solution for problem \eqref{P}. However, in the main result of this paper, stated below, we prove that problem \eqref{P} admits infinitely many weak solutions.

\begin{theorem} \label{main} Let $s\in(0,1)$, $n>2s$, $\Omega$ be an open bounded subset of $\mathbb R^n$ and $r\geq0$. Let $K:\mathbb R^n\setminus\left\{0\right\}\rightarrow\mathbb R^+$ be a function satisfying \eqref{K1}--\eqref{K2}, let $M:\mathbb R^+_0\rightarrow\mathbb R^+_0$ satisfy \eqref{m1}--\eqref{m2} and let $f:\overline{\Omega}\times\mathbb R\rightarrow\mathbb R$ verify \eqref{f1}--\eqref{f2}.

Then, there exists $\overline{\lambda}>0$ such that for any $\lambda\in(0,\overline{\lambda})$ problem~\eqref{P} has infinitely many weak solutions.
\end{theorem}

The proof of Theorem \ref{main} is mainly based on a variational approach combined with some classical results from the genus theory introduced by Krasnoselskii. In order to apply these results, we need a truncation argument which allow us to control from below the Euler--Lagrange functional associated to problem \eqref{P}. Furthermore, as usual in elliptic problems
involving critical nonlinearities, we must pay attention to the lack of compactness at critical
level $L^{2^*}(\Omega)$. To overcome this difficulty, we exploit a concentration-compactness
principle proved, in the fractional framework, in \cite[Theorem 5]{PP}.

The interest in studying problems like \eqref{P} relies not only on mathematical purposes but also on their significance in real models. For example, in the Appendix of paper \cite{FV} the authors  construct a stationary Kirchhoff variational problem which models, as a special
significant case, the nonlocal aspect of the tension arising from nonlocal measurements
of the fractional length of the string.

Several recent papers are focused both on theoretical aspects and applications
related to nonlocal fractional models. Concerning Kirchhoff type problems, in \cite{AFP, FV} the authors prove existence results for problem \eqref{P} when $r=0$.

In \cite{FV} they consider an increasing Kirchhoff function $M$ satisfying \eqref{m1} and, as in our paper, to solve the lack of compactness they use \cite[Theorem 5]{PP}. However, to apply the concentration--compactness principle in \cite{PP}, instead of \eqref{K1}--\eqref{K2} they use the following stronger condition
\begin{equation*}
\mbox{there exists }\theta >0\mbox{ such that }\theta\left|x\right|^{-(n+2s)}\leq K(x)\leq\frac{1}{\theta}\left|x\right|^{-(n+2s)}
\mbox{ for any }x\in\mathbb{R}^{n}\setminus\left\{0\right\}.
\end{equation*}
In our paper, considering the basic notions on the fractional Laplace operator $(-\Delta)^s$, we are able to avoid the requirement of this assumption.

In \cite{AFP}, the authors consider a degenerate Kirchhoff function $M$, namely satisfying $M(0)=0$, for problem \eqref{P}. In this case, the approach based on the application of \cite[Theorem 5]{PP} does not work. Thus, they prove the compactness of the related Euler--Lagrange functional by using the celebrated Brezis \& Lieb lemma. However, as well explained in Section \ref{sec palais}, this approach allow them to prove the existence of just one critical value for the related functional and thus to get just one solution for \eqref{P}.

For existence of multiple solutions it is worth mentioning paper \cite{FBS}, where they study the following subcritical problem
$$ -M(\left\|u\right\|^2)\mathcal L_Ku=\lambda f(x,u)\left[\int_\Omega F(x,u(x))dx\right]^r\quad \mbox{in }\Omega,\qquad u=0\quad\mbox{in }\mathbb{R}^{n}\setminus\Omega
$$
where $M$ could also be in a degenerate setting. While in \cite{CC} and \cite{FS} the authors consider still a critical Kirchhoff problem, similar to \eqref{P}, involving respectively the $p(x)$--Laplacian operator and the classical Laplace operator. As in our paper, in \cite{CC, FBS, FS} the approach is mainly based on a combination of variational and topological techniques. 

Inspired by the above articles, in this paper we would like to investigate the existence of infinitely many solutions for problem \eqref{P}.

The paper is organized as follows. In Section~\ref{sec preliminary} we discuss the variational formulation of the problem and introduce some topological notions.
In Section~\ref{sec palais} we prove the Palais--Smale condition for the functional related to \eqref{P}.
In Section \ref{sec truncation} we introduce a truncation argument for our functional.
In Section \ref{sec main} we prove Theorem~\ref{main}.

\section{Preliminaries}\label{sec preliminary}

\textit{Throughout this paper we assume $s\in(0,1)$, $n>2s$, $\Omega$ is an open bounded subset of $\mathbb R^n$, \eqref{K1}--\eqref{K2}, \eqref{m1}--\eqref{m2} and \eqref{f1}--\eqref{f2} without further mentioning}.

Problem \eqref{P} has a variational structure and the suitable functional space where finding solutions will denoted by $Z$
and it was introduced in \cite{F} in the following way.

First, let $X$ be the linear space of Lebesgue measurable functions $u:\mathbb{R}^n\to\mathbb{R}$ whose restrictions to $\Omega$ belong to $L^2(\Omega)$ and such that
	\[\mbox{the map }
(x,y)\mapsto (u(x)-u(y))^2 K(x-y) \mbox{ is in } L^1\big(Q,dxdy\big)
\]
where $Q=\mathbb{R}^{2n}\setminus\left({\mathcal C}\Omega\times{\mathcal C}\Omega\right)$ and $\mathcal C\Omega=\mathbb{R}^n\setminus\Omega$.
The space $X$ is endowed with the norm
\begin{equation}\label{norma}
\left\|u\right\|_{X}=\Big(\|u\|^2_2+\iint_Q |u(x)-u(y)|^2K(x-y)dxdy\Big)^{1/2}.
\end{equation}
It is easy to see that bounded and Lipschitz functions belong to $X$, thus $X$ is not reduced to
$\left\{0\right\}$ (see \cite{sv3,sv2} for further details on the space $X$).

The functional space $Z$ denotes the closure of $C^{\infty}_{0}(\Omega)$ in $X$.
The scalar product defined for any $\varphi$, $\phi\in Z$ as
\begin{equation}\label{scal}
\langle \varphi, \phi\rangle_Z=\iint_{Q} (\varphi(x)-\varphi(y))(\phi(x)-\phi(y))K(x-y) dxdy
\end{equation}
makes $Z$ a Hilbert space. The norm
\begin{equation}\label{normaz}
\|u\|_{Z}=\Big(\iint_Q |u(x)-u(y)|^2K(x-y)dxdy\Big)^{1/2}
\end{equation}
is equivalent to the usual one defined in \eqref{norma}, as proved in \cite[Lemma~2.1]{F}.
Note that in \eqref{norma}--\eqref{normaz} the integrals can be extended to all $\mathbb{R}^{n}$ and $\mathbb{R}^{2n}$, since  $u=0$ a.e.
in $\mathcal C\Omega$. {\it From now on, in order to simplify the notation, we denote $\langle \cdot, \cdot\rangle_Z$ and
$\|\cdot\|_{Z}$ by $\langle \cdot, \cdot\rangle$ and $\|\cdot\|$, respectively.}

The weak formulation of \eqref{P} is as follows. We say that $u\in Z$ is a weak {\em solution} of~\eqref{P}~if
\begin{equation}\label{wf}
\begin{alignedat}2
M\left(\|u\|^2\right)\langle u, \varphi\rangle&=
\lambda\displaystyle{\int_\Omega f(x, u(x))\varphi(x)dx}\left[\int_\Omega F(x,u(x))dx\right]^r\\
&+\int_\Omega \left|u(x)\right|^{2^*-2}u(x)\varphi(x)dx
\end{alignedat}
\end{equation}
for all $\varphi \in Z$.

Thanks to our assumptions on $\Omega$, $M$, $f$ and $K$, all the integrals in \eqref{wf} are well defined if $u$, $\varphi\in Z$.
Indeed by \eqref{f1} and \eqref{f2}, since $f$ is continuous, we easily get
\begin{equation}\label{f'1}
a_1\left|t\right|^q\leq f(x,t)t\leq a_2\left|t\right|^q\quad\mbox{for any }(x,t)\in\overline{\Omega}\times\mathbb R,
\end{equation}
and
\begin{equation}\label{f'2}
\quad\frac{a_1}{q}\left|t\right|^q\leq F(x,t)\leq\frac{a_2}{q}\left|t\right|^q\quad\mbox{for any }(x,t)\in\overline{\Omega}\times\mathbb R,
\end{equation}
hence $F(x,u(x))\geq0$ for a.a. $x\in\Omega$ in \eqref{wf}.
We also point out that the odd part of function $K$
gives no contribution to the integral of the left-hand side of \eqref{wf} (see \cite{AFP} for details).
Therefore, it would be not restrictive to assume that $K$ is even.

According to the variational nature, weak solutions of \eqref{P} correspond
to critical points of the associated Euler{Lagrange functional $J_\lambda: Z\rightarrow\mathbb R$
defined by
$$\mathcal J_\lambda(u)=\frac{1}{2}\mathscr{M}(\left\|u\right\|^2)-\frac{\lambda}{r+1}\left[\int_\Omega F(x, u(x))dx\right]^{r+1}
-\frac{1}{2^*}\|u\|^{2^*}_{2^*}$$
for any $u\in Z$.

In order to prove the multiplicity result stated in Theorem \ref{main}, we will use some topological results introduced by Krasnoselskii in \cite{K}.
For sake of completeness and for reader's convenience, we recall here some basic notions on the Krasnoselskii's genus.
Let $E$ be a Banach space and let us denote by $\mathcal A$ the class of all closed subsets
$A\subset E\setminus\left\{0\right\}$ that are symmetric with respect to the origin, that is, $u\in A$ implies $-u\in A$.

\begin{definition}
Let $A\in\mathcal A$. The Krasnoselkii's genus $\gamma(A)$ of $A$ is defined as being the least positive integer $k$ such that there is an odd mapping $\phi\in C(A,\mathbb R^k)$ such that $\phi(x)\neq0$ for any $x\in A$. If $k$ does not exist we set $\gamma(A)=\infty$. Furthermore, we set $\gamma(\emptyset)=0$.
\end{definition}

In the sequel we will recall only the properties of the genus that will be
used throughout this work. More information on this subject may be found
in the references \cite{AR, C, K}.

\begin{proposition}
Let $E=\mathbb R^n$ and $\partial\Omega$ be the boundary of an open, symmetric and bounded subset $\Omega$ of $\mathbb R^n$  with $0\in\Omega$. Then $\gamma(\partial\Omega)=n$.
\end{proposition}
As immediate consequence we have the following result.
\begin{corollary}\label{cor}
Let $S^{k-1}$ be a $k-1$ dimensional sphere in $\mathbb R^k$. Then, $\gamma(S^{k-1})=k$.
\end{corollary}
\begin{proposition}\label{prop}
If $K\in\mathcal A$, $0\not\in K$ and $\gamma(K)\geq2$, then $K$ has infinitely
many points.
\end{proposition}

\section{The Palais--Smale condition}\label{sec palais}
In this section we discuss a compactness property for the functional $\mathcal J_\lambda$, given by the Palais--Smale condition. For this, in order to overcome the lack of compactness due to the presence of the critical term we exploit a concentration--compactness principle, introduced in the fractional framework in \cite{PP}. However, the Palais-Smale condition for $\mathcal J_\lambda$ does not hold true at any level, but just under a suitable threshold, which depends also on the best fractional critical Sobolev constant defined by
\begin{equation}\label{sobolev}
S=\inf_{\substack{v\in H^s(\mathbb R^n)\\
v\not=0}}\displaystyle\frac{\iint_{\mathbb R^{2n}}\frac{\left|v(x)-v(y)\right|}{\left|x-y\right|^{n+2s}}}{(\int_{\mathbb R^n} \left|v(x)\right|^{2^*}dx)^{2/2^*}}.
\end{equation}

Before proving this compactness condition, we introduce the following positive constants which will help us for a better explanation
\begin{equation}\label{constants}
\begin{array}{ll}
k_1=\displaystyle\left(\frac{1}{\sigma}-\frac{1}{2^*}\right),\quad & k_2=\displaystyle\left(\frac{1}{r+1}\left(\frac{a_2}{q}\right)^{r+1}-\frac{q}{\sigma}\left(\frac{a_1}{q}\right)^{r+1}\right)\left|\Omega\right|^{\frac{(2^*-q)(r+1)}{2^*}}, \\
k_3=k_1\displaystyle\left[\frac{\theta\, m_0\,S}{c(n,s)}\right]^{n/2s},\quad &k_4=k_1\displaystyle\left(\frac{q(r+1)k_2}{2^*k_1}\right)^{\frac{2^*}{2^*-q(r+1)}}\left(\frac{2^*}{q(r+1)}-1\right).
\end{array}
\end{equation}
The main result of this section is the following theorem.

\begin{lemma}\label{palais}
Let $\left\{u_j\right\}_{j\in\mathbb N}$ be a bounded sequence in $Z$ verifying
\begin{equation}\label{psc}
\mathcal J_\lambda(u_j)\rightarrow c\quad\mbox{and}\quad\mathcal J'_\lambda(u_j)\rightarrow 0\quad\mbox{as }j\rightarrow\infty,
\end{equation}
with
\begin{equation}\label{clevel}
c<k_3-\lambda^{\frac{2^*}{2^*-q(r+1)}}k_4.
\end{equation}

Then, there exists $\lambda_0>0$ such that for any $\lambda\in(0,\lambda_0)$ we have that, up to a subsequence, $\left\{u_j\right\}_{j\in\mathbb N}$ is strongly convergent in $Z$.
\end{lemma}
\begin{proof}

Since the sequence $\left\{u_j\right\}_{j\in\mathbb N}$ is bounded in $Z$, by applying \cite[Lemma~2.1]{F} and \cite[Theorem 4.9]{B}, there exists $u\in Z$ such that, up to a subsequence, it follows that
\begin{equation}\label{convergences}
\begin{array}{ll}
u_j\rightharpoonup u\text{ in }Z\mbox{ and in } L^{2^*}(\Omega),\quad & \left\|u_j\right\|\rightarrow \alpha, \\
u_j\to u\text{ in }L^{q}(\Omega)\mbox{ and in } L^2(\Omega),\quad &u_j\to u\text{ a.e.  in }\Omega,\quad|u_j|\le h\text{ a.e.  in }\Omega
\end{array}
\end{equation}
for some $h\in L^q(\Omega)\cap L^2(\Omega)$.

Now, we claim that 
\begin{equation}\label{claim}
\left\|u_j\right\|^2\rightarrow\left\|u\right\|^2\quad\mbox{as }j\rightarrow\infty,
\end{equation}
which clearly implies that $u_j\rightarrow u$ in $Z$ as $j\rightarrow\infty$.
By \cite[Lemma~2.1]{F} we know that $\left\{u_j\right\}_{j\in\mathbb N}$ is also bounded in $H^s_0(\Omega)$, which is the closure of $C^\infty_0(\Omega)$ in $H^s(\Omega)$ (see also \cite{VP} for further details).
So, by Phrokorov's Theorem (see \cite[Theorem 8.6.2]{Bo}) we may suppose that there exist two positive measures $\mu$ and $\nu$ on $\mathbb{R}^n$ such that
\begin{equation}\label{convergenza misure}
\left|(-\Delta)^{s/2} u_j(x)\right|^2 dx\stackrel{*}{\rightharpoonup}\mu\quad\mbox{and}\quad\left|u_j(x)\right|^{2^*}dx\stackrel{*}{\rightharpoonup}\nu\quad\mbox{in }\mathcal M(\mathbb R^N).
\end{equation}
Moreover, by \cite[Theorem~5]{PP} we obtain an at most countable set of distinct points $\left\{x_i\right\}_{i\in J}$, non--negative numbers $\left\{\nu_i\right\}_{i\in J}$, $\left\{\mu_i\right\}_{i\in J}$ and a positive measure $\widetilde{\mu}$ with $Supp\, \widetilde{\mu}\subset\overline{\Omega}$ such that
\begin{equation}\label{numu}
\nu=\left|u(x)\right|^{2^*}dx+\sum_{i\in J} \nu_i\delta_{x_i},\quad\mu=\left|(-\Delta)^{s/2} u(x)\right|^2 dx+\widetilde{\mu}+\sum_{i\in J} \mu_i\delta_{x_i},
\end{equation}
and
\begin{equation}\label{stima}
\nu_i\leq S^{-2^*/2}\mu^{2^*/2}_{i}
\end{equation}
where $S$ is the best Sobolev constant given in \eqref{sobolev}.
Now, in order to prove \eqref{claim} we proceed by steps.

\begin{step}
Fix $i_0\in J$. Then, we prove that either $\nu_{i_0}=0$ or
\begin{equation}\label{4.6}
\nu_{i_0}\geq\displaystyle\left[\frac{\theta\, m_0\,S}{c(n,s)}\right]^{n/2s}.
\end{equation}
\end{step}

Let $\psi\in C^{\infty}_{0}(\mathbb{R}^n,[0,1])$ be such that $\psi\equiv 1$ in $B(0,1)$ and $\psi\equiv0$ in $\mathbb{R}^n\setminus B(0,2)$. For any $\delta>0$ we set $\psi_{\delta,i_0}(x)=\psi((x-x_{i_0})/\delta)$. Clearly
$\left\{\psi_{\delta,i_0} u_j\right\}_{j\in\mathbb N}$ is bounded in $Z$, and so by \eqref{psc} it follows that $\left\langle \mathcal J'_\lambda(u_j),\psi_{\delta,i_0} u_j\right\rangle\rightarrow 0$ as $j\rightarrow\infty$. From this, by applying also \eqref{K2} we get
\begin{equation}\label{that is}
\begin{alignedat}3
o_j(1)&+\lambda\left[\int_\Omega F(x,u_j(x))dx\right]^r\int_\Omega f(x, u_j(x))\psi_{\delta,i_0}(x)u_j(x)dx\\
&+\int_\Omega \left|u_j(x)\right|^{2^*}\psi_{\delta,i_0}(x)dx\\
&\geq\theta M(\left\|u_j\right\|^2)\iint_{\mathbb{R}^{2n}}\frac{\big(u_j(x)-u_j(y)\big)\big(\psi_{\delta,i_0}(x)u_j(x)-\psi_{\delta,i_0}(y)u_j(y)\big)}
{\left|x-y\right|^{n+2s}}\,dxdy
\end{alignedat}
\end{equation}
as $j\rightarrow\infty$. 

By \cite[Proposition 3.6]{VP} we know that for any $v\in C^\infty_0(\Omega)$
$$\iint_{\mathbb{R}^{2n}} \frac{|v(x)-v(y)|^2}{\left|x-y\right|^{n+2s}}\, dxdy=\frac{1}{c(n,s)}\int_{\mathbb{R}^n}\left|(-\Delta)^{s/2}v(x)\right|^2dx,$$
with $c(n,s)$ the dimensional constant defined in \eqref{cns} and, by taking derivative of the above equality, for any $v, w\in C^\infty_0(\Omega)$ we obtain
\begin{equation}\label{deriv}
\iint_{\mathbb{R}^{2n}} \frac{(v(x)-v(y))(w(x)-w(y))}{\left|x-y\right|^{n+2s}}dxdy=\frac{1}{c(n,s)}\int_{\mathbb{R}^n}(-\Delta)^{s/2}v(x)(-\Delta)^{s/2}w(x)dx.
\end{equation}
Furthermore, for any $v, w\in C^\infty_0(\Omega)$ we have
\begin{equation}\label{prod}
(-\Delta)^{s/2}(v w)=v(-\Delta)^{s/2}w+w(-\Delta)^{s/2}v-2I_{s/2}(v,w)
\end{equation}
where the last term is defined, in the principal value sense, as follows
	\[I_{s/2}(v,w)(x)=P.V.\int_{\mathbb{R}^n}\frac{(v(x)-v(y))(w(x)-w(y))}{\left|x-y\right|^{n+s}}dy
\]
for any $x\in\mathbb R^n$.
Thus, by \eqref{deriv} and \eqref{prod} the integral in the right--hand side of \eqref{that is} becomes
\begin{equation}\label{that2}
\begin{alignedat}2
\iint_{\mathbb{R}^{2n}}&\frac{\big(u_j(x)-u_j(y)\big)\big(\psi_{\delta,i_0}(x)u_j(x)-\psi_{\delta,i_0}(y)u_j(y)\big)}{\left|x-y\right|^{n+2s}}\,dxdy\\
&=\frac{1}{c(n,s)}\int_{\mathbb{R}^n}u_j(x)(-\Delta)^{s/2}u_j(x)(-\Delta)^{s/2}\psi_{\delta,i_0}(x)dx\\
&\quad+\frac{1}{c(n,s)}\int_{\mathbb{R}^n}\left|(-\Delta)^{s/2}u_j(x)\right|^2\psi_{\delta,i_0}(x)dx\\
&\quad-\frac{2}{c(n,s)}\int_{\mathbb{R}^n}(-\Delta)^{s/2}u_j(x)\int_{\mathbb{R}^n}\frac{(u_j(x)-u_j(y))(\psi_{\delta,i_0}(x)-\psi_{\delta,i_0}(y))}{\left|x-y\right|^{n+s}}dxdy.
\end{alignedat}
\end{equation}
By \cite[Lemmas 2.8 and 2.9]{BS} we have
\begin{equation}
\lim_{\delta\rightarrow0}\lim_{j\rightarrow\infty}\left|\int_{\mathbb{R}^n}u_j(x)(-\Delta)^{s/2}u_j(x)(-\Delta)^{s/2}\psi_{\delta,i_0}(x)dx\right|=0
\end{equation}
and
\begin{equation}\label{colorado}
\lim_{\delta\rightarrow0}\lim_{j\rightarrow\infty}\left|\int_{\mathbb{R}^n}(-\Delta)^{s/2}u_j(x)\int_{\mathbb{R}^n}\frac{(u_j(x)-u_j(y))(\psi_{\delta,i_0}(x)-\psi_{\delta,i_0}(y))}{\left|x-y\right|^{n+s}}dxdy\right|=0.
\end{equation}
Thus, by combining \eqref{that2}--\eqref{colorado} and \eqref{convergenza misure}--\eqref{numu} we get
\begin{equation}\label{that3}
\lim_{\delta\rightarrow0}\lim_{j\rightarrow\infty}\left[\iint_{\mathbb{R}^{2n}}
\frac{\big(u_j(x)-u_j(y)\big)\big(\psi_{\delta,i_0}(x)u_j(x)-\psi_{\delta,i_0}(y)u_j(y)\big)}{\left|x-y\right|^{n+2s}}dxdy\right]\geq\frac{1}{c(n,s)}\mu_{i_0}.
\end{equation}

Moreover by \eqref{f'2}, for any $j\in\mathbb N$ we get
\[\frac{a_1}{q}\left\|u_j\right\|^q_q\leq \int_\Omega F(x,u_j(x))dx\leq \frac{a_2}{q}\left\|u_j\right\|^q_q,
\]
and since $\left\{u_j\right\}_{j\in\mathbb N}$ is bounded in $Z$ and $L^q(\Omega)$ we conclude that
\begin{equation}\label{bound}
\mbox{the sequence }\left\{\int_\Omega F(x,u_j(x))dx\right\}_{j\in\mathbb N}\mbox{ is bounded in }\mathbb R.
\end{equation}
While, by \eqref{f'1} and the Dominated Convergence Theorem we get
\begin{equation*}
\int_{B(x_{i_0},\delta)} f(x, u_j(x))u_j(x)\psi_{\delta,i_0}(x)dx\rightarrow\int_{B(x_{i_0},\delta)} f(x, u(x))u(x)\psi_{\delta,i_0}(x)dx\quad\mbox{as }j\rightarrow\infty,
\end{equation*}
and so by sending $\delta\rightarrow0$ we observe that
\begin{equation}\label{term f}
\lim_{\delta\rightarrow0}\lim_{j\rightarrow\infty}\int_{B(x_{i_0},\delta)} f(x, u_j(x))u_j(x)\psi_{\delta,i_0}(x)dx=0.
\end{equation}
Furthermore, by \eqref{convergenza misure} it follows that
	\[\int_\Omega \left|u_j(x)\right|^{2^*}\psi_{\delta,i_0}(x)dx\rightarrow\int_\Omega \psi_{\delta,i_0}(x)d\nu\quad\mbox{as }j\rightarrow\infty,
\]
and by combining this last formula with \eqref{that is}, \eqref{that3} and \eqref{term f}, using also \eqref{convergenza misure}, we obtain
\begin{equation}\label{nui0}
\nu_{i_0}\geq \frac{\theta M(\alpha^2)}{c(n,s)}\mu_{i_0},
\end{equation}
since $M(\left\|u_j\right\|^2)\rightarrow M(\alpha^2)$ as $j\rightarrow\infty$, by continuity of $M$ and \eqref{convergences}.
Thus, by \eqref{m1}, \eqref{stima} and \eqref{nui0} we have that either $\nu_{i_0}=0$ or $\nu_{i_0}$ verifies \eqref{4.6}.

\begin{step}
Estimate \eqref{4.6} can not occur, hence $\nu_{i_0}=0$.
\end{step}
By contradiction we assume that \eqref{4.6} holds true.
By \eqref{psc} we have
\begin{equation}\label{4.7}
c=\lim_{j\rightarrow\infty}\left(\mathcal J_\lambda(u_j)-\frac{1}{\sigma} \left\langle \mathcal J'_\lambda(u_j),u_j\right\rangle\right).
\end{equation}
Moreover, by \eqref{m2}, \eqref{f'1} and \eqref{f'2} we have
\begin{equation}\label{4.8}
\begin{alignedat}5
\mathcal J_\lambda(u_j)&-\frac{1}{\sigma}\left\langle \mathcal J'_\lambda(u_j),u_j\right\rangle\\
&\geq\frac{1}{2}\mathscr{M}(\left\|u_j\right\|^2)-\frac{1}{\sigma}M(\left\|u_j\right\|^2)\left\|u_j\right\|^2-\frac{\lambda}{r+1}\left[\int_\Omega F(x,u_j(x))dx\right]^{r+1}\\
&\quad+\frac{\lambda}{\sigma}\left[\int_\Omega F(x,u_j(x))dx\right]^{r}\int_\Omega f(x,u_j(x))u_j(x)dx\\
&\quad+\left(\frac{1}{\sigma}-\frac{1}{2^*}\right)\int_\Omega\left|u_j(x)\right|^{2^*}dx\\
&\geq\lambda\left(\frac{1}{\sigma}\left(\frac{a_1}{q}\right)^ra_1-\frac{1}{r+1}\left(\frac{a_2}{q}\right)^{r+1}\right)\left[\int_\Omega\left|u_j(x)\right|^q dx\right]^{r+1}\\
&\quad+\left(\frac{1}{\sigma}-\frac{1}{2^*}\right)\int_\Omega\psi_{\delta,i_0}(x)\left|u_j(x)\right|^{2^*}dx,
\end{alignedat}
\end{equation}
since also $0\leq\psi_{\delta,i_0}(x)\leq 1$.
By combining \eqref{4.7} and \eqref{4.8}, using also \eqref{convergenza misure}, we get
	\[c\geq\lambda\left(\frac{q}{\sigma}\left(\frac{a_1}{q}\right)^{r+1}-\frac{1}{r+1}\left(\frac{a_2}{q}\right)^{r+1}\right)\left[\int_\Omega\left|u(x)\right|^q dx\right]^{r+1}+\left(\frac{1}{\sigma}-\frac{1}{2^*}\right)\int_\Omega\psi_{\delta,i_0}(x)d\nu,
\]
from which, by sending $\delta\rightarrow0$ and by using H\" older inequality and \eqref{4.6}, it follows that
\begin{equation*}
\begin{alignedat}2
c&\geq\lambda\left(\frac{q}{\sigma}\left(\frac{a_1}{q}\right)^{r+1}-\frac{1}{r+1}\left(\frac{a_2}{q}\right)^{r+1}\right)\left[\left|\Omega\right|^{\frac{2^*-q}{2^*}}\left\|u\right\|^q_{2^*}\right]^{r+1}\\
	&\quad+\left(\frac{1}{\sigma}-\frac{1}{2^*}\right)\left\|u\right\|^{2^*}_{2^*}+\left(\frac{1}{\sigma}-\frac{1}{2^*}\right)\displaystyle\left[\frac{\theta\, m_0\,S}{c(n,s)}\right]^{n/2s},
\end{alignedat}
\end{equation*}
and by \eqref{constants}
\begin{equation}\label{contradict}
c\geq-\lambda k_2\left\|u\right\|^{q(r+1)}_{2^*}+k_1\left\|u\right\|^{2^*}_{2^*}+k_3.
\end{equation}
Now, it is easy to see that function $g(t)=k_1t^{2^*}-\lambda k_2 t^{q(r+1)}$ attains its absolute minimum at point
	\[t_0=\left(\lambda\,\frac{q(r+1)k_2}{2^*k_1}\right)^{\frac{1}{2^*-q(r+1)}}>0,
\]
and after calculations
	\[g(t_0)=k_1\left(\lambda\,\frac{q(r+1)k_2}{2^* k_1}\right)^{\frac{2^*}{2^*-q(r+1)}}-\lambda k_2\left(\lambda\,\frac{q(r+1)k_2}{2^* k_1}\right)^{\frac{q(r+1)}{2^*-q(r+1)}}=-\lambda^{\frac{2^*}{2^*-q(r+1)}}k_4.
\]
Thus, by \eqref{contradict} we conclude that 
	\[c\geq k_3-\lambda^{\frac{2^*}{2^*-q(r+1)}}k_4,
\]
which contradicts \eqref{clevel}. So, $\nu_{i_0}=0$.

\begin{step}
Claim \eqref{claim} holds true.
\end{step}
By considering that $i_0$ was arbitrary, we deduce that $\nu_i=0$ for any $i\in J$. As a consequence, from also \eqref{convergenza misure} and \eqref{numu} it follows that
$u_j\to u$ in $L^{2^*}(\Omega)$ as $j\rightarrow\infty$.
Since $\left\{u_j\right\}_{j\in\mathbb N}$ is bounded in $Z$, by \eqref{psc} it follows that $\left\langle \mathcal J'_\lambda(u_j),u_j-u\right\rangle\rightarrow 0$ as $j\rightarrow\infty$, that is
\begin{equation}\label{that4}
\begin{alignedat}2
M(\left\|u_j\right\|^2)\left\langle u_j,u_j-u\right\rangle&-\lambda\left[\int_\Omega F(x,u_j(x))dx\right]^r\int_\Omega f(x, u_j(x))(u_j(x)-u(x))dx\\
&+\int_\Omega \left|u_j(x)\right|^{2^*-2}u_j(x)(u_j(x)-u(x))dx=o_j(1)\quad\mbox{as }j\rightarrow\infty.
\end{alignedat}
\end{equation}
By \eqref{f2}, \eqref{convergences} and the Dominated Convergence Theorem we get
\begin{equation}
\left|\int_\Omega f(x,u_j(x))(u_j(x)-u(x))dx\right|\rightarrow 0\quad\mbox{as }j\rightarrow\infty,
\end{equation}
while by considering H\" older inequality
\begin{equation}\label{that4.1}
\left|\int_\Omega \left|u_j(x)\right|^{2^*-2}u_j(x)(u_j(x)-u(x))dx\right|\rightarrow0\quad\mbox{as }j\rightarrow\infty.
\end{equation}
So, by \eqref{that4}--\eqref{that4.1} and \eqref{bound} we have
\begin{equation}
M(\alpha^2)(\left\|u_j\right\|^2-\left\langle u_j,u\right\rangle)\rightarrow0\quad\mbox{as }j\rightarrow\infty,
\end{equation}
recalling that $M(\left\|u_j\right\|^2)\rightarrow M(\alpha^2)$ as $j\rightarrow\infty$.
From this, remembering that $u_j\rightharpoonup u$ in $Z$ and considering that $M(\alpha^2)>0$ by \eqref{m1}, we conclude the proof of claim \eqref{claim}.
\end{proof}

In \eqref{clevel} we stress the presence of $m_0$, given by \eqref{m1}. For this, the same approach used in Lemma \ref{palais} does not work in the degenerate framework. Indeed, if $M$ is a continuous function satisfying $M(0)=0$, we could still have positive boundedness from below (see for example condition $(M_2)$ in \cite{AFP}), but the bound is no longer uniform as in \eqref{m1}. However, problem \eqref{P} with a degenerate Kirchhoff function $M$ could admit a non--trivial solution, by proceeding as in the proof of \cite[Theorem 1.1]{AFP}. For this, we could prove a Palais--Smale condition just at a suitable level $c_\lambda$ defined in \cite{AFP}, by strongly exploiting the asymptotic condition $\displaystyle\lim_{\lambda\rightarrow\infty}c_\lambda=0$, proved in \cite[Lemma 3.3]{AFP}, to overcome the lack of compactness due to the presence of a critical term.

\section{A truncation argument}\label{sec truncation}

We note that our functional $\mathcal J_\lambda$ is not bounded from below in $Z$.
Indeed, by \eqref{m2} for any $\varepsilon>0$ there exists $\delta=\delta(\varepsilon)=\mathscr M(\varepsilon)/\varepsilon^{\sigma/2}$ such that
\begin{equation}\label{m'2}
\mathscr M(t)\leq\delta t^{\sigma/2}\quad\mbox{for any }t\geq\varepsilon.
\end{equation}
So, by fixing $\varepsilon>0$ in the above inequality and by using \eqref{f'2} we see that for any $u\in Z$
	\[\mathcal J_\lambda(tu)\leq t^\sigma\frac{\delta}{2}\left\|u\right\|^\sigma-t^{q(r+1)}\frac{\lambda}{r+1}\left(\frac{a_1}{q}\right)^{r+1}\left\|u\right\|^{q(r+1)}_q-t^{2^*}\frac{1}{2^*}\left\|u\right\|^{2^*}_{2^*}\rightarrow-\infty\quad\mbox{as }t\rightarrow\infty,
\]
since $q(r+1)<2\leq\sigma<2^*$.

For this in the sequel we introduce a truncation like in \cite{AP}, to get a special lower bound which will be worth to construct critical values for $\mathcal J_\lambda$.
By \eqref{m1}, \eqref{f'2} and the fractional Sobolev inequality proved in \cite[Theorem 6.5]{VP}
	\[\mathcal J_\lambda(u)\geq\frac{m_0}{2}\left\|u\right\|^2-\frac{\lambda}{r+1}\left(\frac{a_2}{q}\right)^{r+1}\frac{1}{S^{r+1}_q}\left\|u\right\|^{q(r+1)}-\frac{1}{2^*S_{2^*}}\left\|u\right\|^{2^*}=\mathcal G_\lambda(\left\|u\right\|)
\]
where we denote
	\[\mathcal G_\lambda(t)=\frac{m_0}{2}t^2-\frac{\lambda}{r+1}\left(\frac{a_2}{q}\right)^{r+1}\frac{1}{S^{r+1}_q}t^{q(r+1)}-\frac{1}{2^*S_{2^*}}t^{2^*}.
\]
Now, we can take $R_1>0$ sufficiently small such that
	\[\frac{m_0}{2}R^2_1-\frac{1}{2^*S_{2^*}}R^{2^*}_1>0,
\]
and we define
\begin{equation}\label{stella}
\lambda^*=\frac{r+1}{2}\left(\frac{q}{a_2}\right)^{r+1}\frac{S^{r+1}_q}{R^{q(r+1)}_1}\left(\frac{m_0}{2}R^2_1-\frac{1}{2^*S_{2^*}}R^{2^*}_1\right),
\end{equation}
so that $\mathcal G_{\lambda^*}(R_1)>0$. From this, we consider
	\[R_0=\max\left\{t\in(0,R_1):\,\,\mathcal G_{\lambda^*}(t)\leq0\right\}.
\]
Since by $q(r+1)<2$ we have $\mathcal G_{\lambda}(t)\leq0$ for t near to 0 and since also $\mathcal G_{\lambda^*}(R_1)>0$, it easily follows that $\mathcal G_{\lambda^*}(R_0)=0$.

We can choose $\phi\in C^\infty_0([0,\infty),[0,1])$ such that $\phi(t)=1$ if $t\in[0,R_0]$ and $\phi(t)=0$ if $t\in[R_1,\infty)$. So, we consider the truncated functional
$$\mathcal I_\lambda(u)=\frac{1}{2}\mathscr{M}(\left\|u\right\|^2)-\frac{\lambda}{r+1}\left[\int_\Omega F(x, u(x))dx\right]^{r+1}
-\phi(\left\|u\right\|)\frac{1}{2^*}\|u\|^{2^*}_{2^*}.$$
It immediately follows that $\mathcal I_\lambda(u)\rightarrow\infty$ as $\left\|u\right\|\rightarrow\infty$. Hence, $\mathcal I_\lambda$ is coercive and bounded from below.

Now, we prove a local Palais--Smale and a topological result for the truncated functional $\mathcal I_\lambda$.

\begin{lemma}\label{localpalais}
There exists $\overline{\lambda}>0$ such that for any $\lambda\in(0,\overline{\lambda})$
\begin{enumerate}
\item[$(i)$] if $\mathcal I_\lambda(u)\leq0$ then $\left\|u\right\|<R_0$ and also $\mathcal J_\lambda(v)=\mathcal I_\lambda(v)$ for any $v$ in a sufficiently small neighborhood of $u$;
\item[$(ii)$]
$\mathcal I_\lambda$ satisfies a local Palais--Smale condition for $c\leq0$.
\end{enumerate} 
\end{lemma}
\begin{proof}
Considering $\lambda_0$ and $\lambda^*$ given respectively by Lemma \ref{palais} and \eqref{stella}, we choose $\overline{\lambda}$ sufficiently small such that $\overline{\lambda}\leq\min\left\{\lambda_0,\lambda^*\right\}$ and
	\[0<k_3-\overline{\lambda}\,^{\frac{2^*}{2^*-q(r+1)}}k_4,
\]
with $k_3, k_4$ defined as in \eqref{constants}. Let $\lambda<\overline{\lambda}$.

For proving $(i)$ we assume that $\mathcal I_\lambda(u)\leq0$. When $\left\|u\right\|\geq R_1$, by using \eqref{m1}, \eqref{f'2}, \cite[Theorem 6.5]{VP} and $\lambda<\lambda^*$, we see that
	\[\mathcal I_\lambda(u)\geq\frac{m_0}{2}\left\|u\right\|^2-\frac{\lambda^*}{r+1}\left(\frac{a_2}{q}\right)^{r+1}\frac{1}{S^{r+1}_q}\left\|u\right\|^{q(r+1)}>0
\]
where the last inequality follows by $q(r+1)<2$ and because by $\mathcal G_{\lambda^*}(R_1)>0$ we have
	\[\frac{m_0}{2}R^2_1-\frac{\lambda^*}{r+1}\left(\frac{a_2}{q}\right)^{r+1}\frac{1}{S^{r+1}_q}R^{q(r+1)}_1>0.
\]
Thus, we get the contradiction $0\geq\mathcal I_\lambda(u)>0$.
When $\left\|u\right\|< R_1$, since $\phi(t)\leq 1$ for any $t\in[0,\infty)$ and $\lambda<\lambda^*$, we have
	\[0\geq\mathcal I_\lambda(u)\geq\mathcal G_\lambda(\left\|u\right\|)\geq\mathcal G_{\lambda^*}(\left\|u\right\|),
\]
and this yields $\left\|u\right\|\leq R_0$, by definition of $R_0$.
Furthermore, for any $u\in B(0,R_0/2)$ we have $\mathcal I_\lambda(u)=\mathcal J_\lambda(u)$.

To prove a local Palais--Smale condition for $\mathcal I_\lambda$ at level $c\leq0$, we first observe that any Palais--Smale sequences for $\mathcal I_\lambda$ must be bounded, since $\mathcal I_\lambda$ is coercive.
So, since $\lambda<\lambda_0$ and
	\[0<k_3-\overline{\lambda}\,^{\frac{2^*}{2^*-q(r+1)}}k_4<k_3-\lambda^{\frac{2^*}{2^*-q(r+1)}}k_4,
\]
by Lemma \ref{palais} we have a local Palais--Smale condition for $\mathcal J_\lambda\equiv\mathcal I_\lambda$ at any level $c\leq 0$.
\end{proof}

Here we briefly recall the following eigenvalue problem
\begin{equation}\label{autovalori}
-\mathcal L_K u=\lambda u\quad\mbox{in }\Omega,\qquad u=0\quad\mbox{in }\mathbb R^{n}\setminus\Omega.
\end{equation}
The spectral theory related to problem \eqref{autovalori} will be worth to get the next technical lemma.
By \cite[Proposition 2.2]{F} we know that operator $-\mathcal L_K$ possesses a divergent sequence of positive eigenvalues
	\[\lambda_1<\lambda_2\leq\ldots\leq\lambda_k\leq\lambda_{k+1}\leq\ldots
\]
In the sequel we will denote by $e_k$ the eigenfunction related to the eigenvalue $\lambda_k$, for any $k\in\mathbb N$.

\begin{lemma}\label{genus}
For any $\lambda>0$ and $k\in\mathbb N$, there exists $\varepsilon=\varepsilon(\lambda, k)>0$ such that 
	\[\gamma(\mathcal I_\lambda^{-\varepsilon})\geq k,
\]
where $\mathcal I_\lambda^{-\varepsilon}=\left\{u\in Z:\,\,\mathcal I_\lambda(u)\leq-\varepsilon\right\}$.
\end{lemma}
\begin{proof}
Fix $\lambda>0$, $k\in\mathbb N$ and let us consider $\mathbb H_k=\mbox{span}\left\{e_1,\ldots,e_k\right\}$ the linear subspace of $Z$ generated by the first $k$ eigenfunctions of problem \eqref{autovalori}. Since $\mathbb H_k$ is finite dimensional, there exists a positive constant $c(k)$ such that
	\[c(k)\left\|u\right\|^q\leq\left\|u\right\|^q_q,
\]
for any $u\in \mathbb H_k$.
So by using also \eqref{f'2}, for any $u\in\mathbb H_k$ with $\left\|u\right\|\leq R_0$ we get
\begin{equation*}
\begin{alignedat}2
\mathcal I_\lambda(u)&\leq \frac{m^*}{2}\left\|u\right\|^2-\frac{\lambda}{(r+1)}\left(\frac{a_1}{q}\right)^{r+1}\left\|u\right\|^{q(r+1)}_q-\frac{1}{2^*}\left\|u\right\|^{2^*}_{2^*}\\
&\leq\frac{m^*}{2}\left\|u\right\|^2-\frac{\lambda}{(r+1)}\left(\frac{a_1 c(k)}{q}\right)^{r+1}\left\|u\right\|^{q(r+1)},
\end{alignedat}
\end{equation*}
with $m^*=\max_{\tau\in[0,R_0]} M(\tau)<\infty$, by continuity of $M$.
Finally, let $\rho$ and $R$ be two positive constants with
\begin{equation}\label{raggi}
\rho<R<\min\left\{R_0, \left[\frac{\lambda}{(r+1)}\left(\frac{a_1 c(k)}{q}\right)^{r+1}\frac{2}{m^*}\right]^{\frac{1}{2-q(r+1)}}\right\},
\end{equation}
and let
	\[\mathbb S_k=\left\{u\in\mathbb H_k:\,\,\left\|u\right\|=\rho\right\}.
\]
Of course, $\mathbb S_k$ is homeomorphic to $S^{k-1}$. Moreover for any $u\in\mathbb S_k$
\begin{equation*}
\begin{alignedat}2
\mathcal I_\lambda(u)&\leq\rho^{q(r+1)}\left(\frac{m^*}{2}\rho^{2-q(r+1)}-\frac{\lambda}{(r+1)}\left(\frac{a_1 c(k)}{q}\right)^{r+1}\right)\\
&\leq R^{q(r+1)}\left(\frac{m^*}{2}R^{2-q(r+1)}-\frac{\lambda}{(r+1)}\left(\frac{a_1 c(k)}{q}\right)^{r+1}\right)<0
\end{alignedat}
\end{equation*}
where the last inequality follows by \eqref{raggi}. So we can find a constant $\varepsilon>0$ such that $\mathcal I_\lambda(u)<-\varepsilon$ for any $u\in\mathbb S_k$. Hence $\mathbb S_k\subset\mathcal I_\lambda^{-\varepsilon}$ and by monotonicity of genus and Corollary \eqref{cor}, we get $\gamma(\mathcal I_\lambda^{-\varepsilon})\geq\gamma(\mathbb S_k)=k$.
\end{proof}

\section{Main result}\label{sec main}

Here we define for any $k\in\mathbb N$ the sets
	\[\Gamma_k=\left\{C\subset Z:\,\,C\mbox{ is closed}, C=-C\mbox{ and }\gamma(C)\geq k\right\},
\]
	\[ K_c=\left\{u\in Z:\,\,\mathcal I'_\lambda(u)=0\mbox{ and }\mathcal I_\lambda(u)=c\right\},
\]
and the number $c_k=\inf_{C\in\Gamma_k}\sup_{u\in C}\mathcal I_\lambda(u)$.
Before proving our main result, we state some crucial properties of the family of numbers $\left\{c_k\right\}_{k\in\mathbb N}$.

\begin{lemma}\label{negative}
For any $\lambda>0$ and $k\in\mathbb N$, the number $c_k$ is negative.
\end{lemma}
\begin{proof}
Let $\lambda>0$ and $k\in\mathbb N$.
By Lemma \ref{genus}, there exists $\varepsilon>0$ such that $\gamma(\mathcal I_\lambda^{-\varepsilon})\geq k$. Since also $\mathcal I_\lambda$ is continuous and even, $\mathcal I_\lambda^{-\varepsilon}\in\Gamma_k$. From $\mathcal I_\lambda(0)=0$ we have $0\not\in\mathcal I_\lambda^{-\varepsilon}$. Furthermore $\sup_{u\in\mathcal I_\lambda^{-\varepsilon}}\mathcal I_\lambda(u)\leq-\varepsilon$. In conclusion, remembering also that $\mathcal I_\lambda$ is bounded from below, we get
	\[-\infty<c_k=\inf_{C\in\Gamma_k}\sup_{u\in C}\mathcal I_\lambda(u)\leq\sup_{u\in\mathcal I_\lambda^{-\varepsilon}}\mathcal I_\lambda(u)\leq-\varepsilon<0.
\]
\end{proof}

\begin{lemma}\label{deformation}
Let $\lambda\in(0,\overline{\lambda})$, where $\overline{\lambda}$ is the constant given in Lemma \ref{localpalais}, and let $k\in\mathbb N$. If $c=c_k=c_{k+1}=\ldots=c_{k+l}$ for some $l\in\mathbb N$, then
	\[\gamma(K_c)\geq l+1.
\]
In particular, each $c_k$ is a critical value for $\mathcal I_\lambda$.

\end{lemma}
\begin{proof}
Let $\lambda\in(0,\overline{\lambda})$ and $k,l\in\mathbb N$.
Since from Lemma \ref{negative} we have $c=c_k=c_{k+1}=\ldots=c_{k+l}$ is negative, by Lemma \ref{localpalais} the functional $\mathcal I_\lambda$ satisfies the Palais--Smale condition in $K_c$ and it easily follows that $K_c$ is compact.

If by contradiction $\gamma(K_c)\leq l$, then there exists a closed and symmetric set $U$, with $K_c\subset U$, such that $\gamma(U)\leq r$. Since $c<0$, we can choose $U\subset\mathcal I_\lambda^0$. By \cite[Theorem 3.4]{Be} we have an odd homeomorphism $\eta:Z\rightarrow Z$ such that 
\begin{equation}\label{4.3'}
\eta(\mathcal I_\lambda^{c+\delta}-U)\subset\mathcal I_\lambda^{c-\delta}
\end{equation}
for some $\delta\in(0,-c)$.
So, it follows that $\mathcal I_\lambda^{c+\delta}\subset\mathcal I_\lambda^0$. By definition of $c=c_{k+l}$ there exists $A\in\Gamma_{k+l}$ such that $\sup_{u\in A}\mathcal I_\lambda(u)<c+\delta$, that is $A\subset\mathcal I_\lambda^{c+\delta}$, and by using also \eqref{4.3'}
\begin{equation}\label{4.3}
\eta(A-U)\subset\eta(\mathcal I_\lambda^{c+\delta}-U)\subset\mathcal I_\lambda^{c-\delta}.
\end{equation}
However, by properties of genus (see \cite{K}) we have
	\[\gamma(\overline{A-U})\geq\gamma(A)-\gamma(U)\geq k,
\]
and by using also monotonicity
	\[\gamma(\eta(\overline{A-U}))\geq\gamma(\overline{A-U})\geq k.
\]
Hence, we get $\eta(\overline{A-U})\in\Gamma_k$ which implies
	\[\sup_{u\in\eta(\overline{A-U})}\mathcal I_\lambda(u)\geq c_k=c,
\]
and this fact contradicts \eqref{4.3}.
\end{proof}

\begin{proof}[\bf{Proof of Theorem \ref{main}}]
Let $\overline{\lambda}$ be the constant given in Lemma \ref{localpalais} and let $\lambda\in(0,\overline{\lambda})$. We can consider two cases.

If $-\infty<c_1<c_2<\ldots<c_k<c_{k+1}<\ldots$ where $\left\{c_k\right\}_{k\in\mathbb N}$ are negative, by Lemma \ref{negative}, and critical points for $\mathcal I_\lambda$, by Lemma \ref{deformation}, then by Lemma \ref{localpalais} we have infinitely many critical points for $\mathcal J_\lambda$. Hence, problem \eqref{P} has infinitely many solutions.

If there exist $k,l\in\mathbb N$ such that $c_k=c_{k+1}=\ldots=c_{k+l}=c$, then $\gamma(K_c)\geq l+1\geq2$ by Lemma \ref{deformation}. So, by Proposition \ref{prop} the set $K_c$ has infinitely many points, which are infinitely many critical points for $\mathcal J_\lambda$ by Lemma \ref{localpalais}. Thus again, problem \eqref{P} has infinitely many solutions.
\end{proof}

\section*{Acknowledgments}
The author is supported by {\em Coordena\c c\~ao de Aperfei\c conamento de pessoal de n\'ivel superior} through the fellowship PNPD--CAPES 33003017003P5.
The author is member of the {\em Gruppo Nazionale per l'Analisi Ma\-tema\-tica, la Probabilit\`a e
le loro Applicazioni} (GNAMPA) of the {\em Istituto Nazionale di Alta Matematica ``G. Severi"} (INdAM).

\end{document}